\documentclass[11pt]{article}
\usepackage[T1]{fontenc}
\usepackage[margin=1in]{geometry}
\usepackage{amsmath,amssymb,amsthm,mathtools}
\usepackage{newtxtext}
\usepackage{newtxmath}
\usepackage{microtype}
\usepackage{enumitem}
\usepackage{booktabs}
\usepackage{authblk}
\usepackage[numbers,sort&compress]{natbib}
\usepackage[hidelinks]{hyperref}

\allowdisplaybreaks[2]
\setlist[enumerate]{leftmargin=2.3em,itemsep=0.2em,topsep=0.3em}
\newtheorem{theorem}{Theorem}[section]
\newtheorem{proposition}[theorem]{Proposition}
\newtheorem{lemma}[theorem]{Lemma}
\newtheorem{corollary}[theorem]{Corollary}
\theoremstyle{definition}
\newtheorem{definition}[theorem]{Definition}
\theoremstyle{remark}
\newtheorem{remark}[theorem]{Remark}
\newcommand{\R}{\mathbb{R}}
\newcommand{\norm}[1]{\left\lVert #1\right\rVert}
\newcommand{\trans}{\mathsf{T}}
\newcommand{\Span}{\operatorname{span}}
\newcommand{\op}{\mathrm{op}}
\DeclareMathOperator{\supp}{supp}

\makeatletter
\renewcommand\AB@authnote[1]{%
  \raisebox{0.75ex}{\scriptsize\normalfont #1}}
\renewcommand\AB@affilnote[1]{%
  \raisebox{0.75ex}{\scriptsize\normalfont #1}\,}
\makeatother
\title{The Minimum Q-Order of BFGS with Exact Line Search Is One}
\author[1]{Benqi Liu}
\author[2]{Chenyi Li}
\author[3]{Zaiwen Wen}
\affil[1]{Beijing International Center for Mathematical Research,
  Peking University, Beijing 100871, China;
  Email: \href{mailto:bqliu@pku.edu.cn}{\nolinkurl{bqliu@pku.edu.cn}}}
\affil[2]{School of Mathematical Sciences, Peking University,
  Beijing 100871, China;
  Email: \href{mailto:lichenyi@stu.pku.edu.cn}{%
    \nolinkurl{lichenyi@stu.pku.edu.cn}}}
\affil[3]{Beijing International Center for Mathematical Research,
  Peking University, Beijing 100871, China;
  Email: \href{mailto:wenzw@pku.edu.cn}{\nolinkurl{wenzw@pku.edu.cn}}}
\date{}
\begin{document}
\maketitle
\begin{abstract}
Powell asked whether the smoothness assumptions underlying classical
superlinear convergence force a fixed power law between adjacent iterates of
exact-line-search variable-metric methods. We answer this question negatively
for BFGS: within the smooth strongly convex setting, the smallest possible
adjacent-iterate Q-order is one, and this boundary is attained by a single
nonterminating run. In every finite dimension at least two, and for any
prescribed radius and Hessian tolerance, we construct an infinitely
differentiable, globally strongly convex objective that equals the standard
quadratic outside the corresponding ball and whose Hessian remains within the
prescribed tolerance of the identity in operator norm. The objective has its
unique minimizer at the origin and identity Hessian there. Exact-line-search
BFGS, initialized with the identity matrix and started inside that ball,
converges Q-superlinearly, yet no fixed power greater than one controls all
sufficiently late adjacent errors.
\end{abstract}

\section{Introduction}
The Broyden--Fletcher--Goldfarb--Shanno (BFGS) method, introduced independently
by Broyden \cite{Broyden1970}, Fletcher \cite{Fletcher1970}, Goldfarb
\cite{Goldfarb1970}, and Shanno \cite{Shanno1970}, is a central quasi-Newton
method for continuous optimization. Its practical appeal comes from combining
curvature information with first-order oracle calls. Classical local analyses
by Powell \cite{Powell1971}, Broyden, Dennis, and Mor\'e
\cite{BroydenDennisMore1973}, Dennis and Mor\'e \cite{DennisMore1974}, and
Nocedal \cite{Nocedal1992} establish superlinear convergence near a
nondegenerate minimizer under appropriate Hessian regularity. A sharper
question remains: do unlimited smoothness and excellent conditioning force
the error of every sufficiently late iterate to be a fixed power of the
preceding error?

This question distinguishes two rate notions that are often grouped under the
term superlinear convergence. Q-superlinear convergence requires only
$$
  \frac{\norm{x_{k+1}-x_\star}}{\norm{x_k-x_\star}}\longrightarrow 0,
$$
whereas adjacent-iterate Q-order $p>1$ requires a constant $C>0$ such that
\begin{equation}\label{eq:intro-qorder}
  \norm{x_{k+1}-x_\star}\le C\norm{x_k-x_\star}^{p}
\end{equation}
for all sufficiently large $k$. Unlike an $n$-step product estimate or an
R-order bound based on a majorizing sequence, \eqref{eq:intro-qorder} controls
every late adjacent pair. It is therefore the natural power-law guarantee at
the single-iteration scale; see Powell \cite{Powell1983}, Schuller
\cite{Schuller1974}, and Ritter \cite{Ritter1980} for the corresponding
classical rate results.

Powell \cite[p.~1538]{Powell1983} asked for the least Q-order of
exact-line-search variable-metric methods near a nondegenerate minimizer when
the objective is infinitely differentiable and uniformly convex. Yuan
\cite{Yuan1984} proved that the infimum is one for uniformly convex $C^2$
objectives and constructed examples of Q-order $1+1/N$, with $N$ arbitrary.
Those examples approach but do not attain the boundary. Taking $N$ to
infinity changes the constructed objective and trajectory, so it does not
produce one run on which every exponent $p>1$ fails. Yuan also asked whether
the class associated with his particular fixed trajectory contains a function
differentiable to all orders. That trajectory-specific realization question
is distinct from the broader rate question studied here. To the best of our
knowledge, no previous construction attains Q-order one on a single
nonterminating trajectory in the $C^\infty$ strongly convex class. Dai's
recent overview \cite{Dai2022Overview} likewise records the sharp classical
conclusion as an infimum-one result.

We close this broader attainment gap. The minimum Q-order is one, and a single
nonterminating BFGS run attains it. The example is not caused by poor
conditioning or irregular behavior away from the solution. Its objective is
globally $C^\infty$ and strongly convex, its Hessian is uniformly as close to
the identity as desired, and it agrees with the standard quadratic outside an
arbitrarily small ball. The run also starts from the standard Hessian
approximation $B_0=I_n$. Thus no adjacent-iterate power order greater than one
can be deduced from the classical exact-line-search assumptions, even after
they are strengthened by near-unit global conditioning and a quadratic tail.

\subsection{Contributions}
The paper makes three contributions. First, it replaces Yuan's limiting family
of orders $1+1/N$ by one fixed problem and one fixed nonterminating trajectory
whose Q-order is exactly one. This proves that the universal lower bound is a
minimum, not only an infimum. Second, the result is robust: it holds in every
dimension $n\ge2$, uses $B_0=I_n$, and persists when the objective is
arbitrarily close to a quadratic in uniform Hessian norm. The nonquadratic
part and the initial point can both be confined to an arbitrarily small ball;
in particular, the Hessian is globally Lipschitz and all derivatives of order
at least three are bounded. Third, the proof develops an identity-initialized
two-dimensional realization criterion and combines it with an alternating
gradient construction. Flat interpolation and convex conjugacy turn the
discrete trajectory into a globally smooth strongly convex objective without
changing its adjacent-iterate order.

For the precise statement, write
$$
Q(x):=\frac12\norm{x}^2 \qquad (x\in\R^n)
$$
for the standard quadratic on $\R^n$.
Throughout, $\norm{\cdot}$ denotes the Euclidean norm on $\R^n$ and
$\norm{\cdot}_{\op}$ the induced operator norm. We write $I_n$ for the
$n\times n$ identity matrix, $B(x,r)$ for the open Euclidean ball centered at
$x$ with radius $r$, and $\overline{B}(x,r)$ for its closure. For a function
$f$, $\supp f$ denotes its support, and $C_c^\infty(\R^n)$ denotes the space
of smooth compactly supported functions. For symmetric matrices, $A\preceq B$
denotes the Loewner order.
\begin{theorem}[Main result]\label{thm:main}
For every finite dimension $n\ge2$, every $\varepsilon>0$, and every $R>0$,
there exist a function $F:\R^n\to\R$ and a point $x_0\in\R^n$ with the
following properties:
\begin{enumerate}[label=\textnormal{(\roman*)}]
  \item $F\in C^\infty(\R^n)$ is globally strongly convex,
  $$
  F-Q\in C_c^\infty(\R^n),\qquad
    \supp(F-Q)\cup\{x_0\}\subset B(0,R),
  $$
  and
  \begin{equation*}
    \sup_{x\in\R^n}\norm{\nabla^2F(x)-I_n}_{\op}<\varepsilon.
  \end{equation*}
  Its unique minimizer is $x_\star=0$, and $\nabla^2F(0)=I_n$.
  \item BFGS initialized with $B_0=I_n$ and using exact line search is well
  defined. Every Hessian approximation $B_k$ remains positive definite, the
  method does not terminate finitely, and its iterates converge to $0$.
  \item The iterates converge Q-superlinearly,
  $$
  \frac{\norm{x_{k+1}}}{\norm{x_k}}\longrightarrow0,
  $$
  while its Q-order at the origin, as defined in
  Definition~\ref{def:qorder}, is one.
\end{enumerate}
Consequently, among nonterminating exact-line-search BFGS sequences that
converge to a nondegenerate minimizer of a $C^\infty$ strongly convex
objective, the smallest possible adjacent-iterate Q-order is one. The
construction attains this value.
\end{theorem}

As a direct consequence, Dixon's exact-line-search equivalence
\cite{Dixon1972Identical} transfers the same iterate sequence, and hence the
same Q-order conclusion, to the convex Broyden class; see
Corollary~\ref{cor:broyden}.

The theorem separates two statements that are sometimes conflated. The BFGS
run is Q-superlinear, so every adjacent error ratio tends to zero. Nevertheless,
no single exponent $p>1$ controls all sufficiently late adjacent errors. The
result concerns exact line search and makes no claim about BFGS with Wolfe or
other inexact line searches. It also does not contradict $n$-step estimates,
R-order bounds, or nonasymptotic function-value rates, which control products,
majorizing sequences, or aggregate progress rather than an adjacent-iterate
power law.

\subsection{Proof strategy}
The proof has two parts: a discrete construction of the desired rate and a
smooth realization of that construction. We first construct a sequence
$\{g_k\}_{k\ge0}$ and candidate BFGS iterates $x_k=g_k+\Delta_k$, where
$\Delta_k\perp g_k$. On a fixed two-dimensional subspace, we design a
parameterized recurrence that enforces both the exact-line-search
orthogonality condition and a BFGS realization criterion. At even indices we
enforce near cancellation; at the subsequent odd indices we retain the
leading term. Every adjacent radius ratio then tends to zero, but the ratio of
successive negative log-radii tends to one along the retention transitions.

The perturbations $\Delta_k$ are flat to all orders and uniformly small
relative to $\norm{g_k}$. We interpolate the prescribed values by pairwise
disjoint $C^\infty$ bump functions. The resulting Hessian perturbation can be
made uniformly as small as desired. Convex conjugacy then yields the required
objective. A final scaling places both the nonquadratic perturbation and the
starting point in an arbitrarily small ball without changing the BFGS matrices
or the Q-order.
\subsection{Relation to prior work}
Prescribed BFGS trajectories and related interpolation arguments appear in
the nonconvex examples of Mascarenhas \cite{Mascarenhas2004,Mascarenhas2007}
and Dai \cite{Dai2013}. The realization criterion used below is adapted from
the convex converse criterion of Lewis and Zhang
\cite[Proposition~3.1]{LewisZhang2015}. Strong convexity in our setting yields
uniqueness, while the extra span condition in
Lemma~\ref{lem:identity-realization} fixes the standard initialization
$B_0=I_n$. Smooth convex interpolation and convex conjugacy are used in a
different context by Bolte and Pauwels \cite{BoltePauwels2022}.

Yuan's construction \cite{Yuan1984} established the infimum-one result in the
$C^2$ class. Its examples had Q-orders $1+1/N$, and its realized objective was
$C^2$. We do not decide whether Yuan's particular fixed trajectory admits a
$C^\infty$ realization. Instead, we use a different scale sequence to attain
Q-order one on a globally $C^\infty$ objective. The construction also fixes
the standard initialization and global conditioning through a flat, compactly
supported smooth realization.

Table~\ref{tab:boundary-comparison} isolates the distinction that matters for
the present paper. An infimum over a family does not provide a member at the
boundary, whereas our theorem produces one fixed problem and one fixed run on
which every power estimate with $p>1$ fails.
\begin{table}[t]
\centering
\caption{From the classical infimum result to boundary attainment.}
\label{tab:boundary-comparison}
\begin{tabular}{@{}p{0.21\textwidth}p{0.34\textwidth}p{0.37\textwidth}@{}}
\toprule
Property & Yuan \cite{Yuan1984} & Theorem~\ref{thm:main} \\
\midrule
Boundary statement
  & The infimum of possible Q-orders is one.
  & The minimum Q-order is one and is attained. \\
Constructed rate
  & Q-order $1+1/N$ for each fixed $N$.
  & Q-order exactly one on a single nonterminating run. \\
Smooth realization
  & A uniformly convex $C^2$ objective; all-order differentiability for the
    specified trajectory was left open.
  & A globally $C^\infty$ strongly convex objective. \\
Global control
  & Uniform convexity.
  & Hessian arbitrarily close to $I_n$ and an exact quadratic tail. \\
\bottomrule
\end{tabular}
\par\smallskip
\begin{minipage}{0.94\textwidth}
\small\emph{Note.}
The comparison concerns established properties of the respective
constructions; it does not assert that Yuan's fixed trajectory has no smoother
realization.
\end{minipage}
\end{table}

Local nonasymptotic analyses of unit-step BFGS were given by Rodomanov and
Nesterov \cite{RodomanovNesterov2022} and Jin and Mokhtari
\cite{JinMokhtari2023}. Jin, Jiang, and Mokhtari
\cite{JinJiangMokhtari2025} more recently obtained global nonasymptotic
function-value bounds for BFGS with exact line search. Those results quantify
aggregate progress and are not adjacent-iterate power estimates. The present
counterexample does not weaken them. Instead, it shows that such aggregate
bounds cannot in general be converted into a uniform Q-order $p>1$, even when
the objective is arbitrarily close to a quadratic.
\paragraph{Organization.}
Section~\ref{sec:setting} defines the notion of Q-order used here and proves
the identity-initialized BFGS realization criterion. Section~\ref{sec:discrete}
constructs the alternating gradient sequence. Section~\ref{sec:smooth}
realizes this sequence by a compactly supported smooth perturbation and records
the effects of convex conjugacy and scaling. Section~\ref{sec:proof} then
proves the main theorem. Section~\ref{sec:conclusion} concludes the paper.
\section{Q-order and identity-initialized BFGS realization}
\label{sec:setting}
We use the Hessian form of BFGS. Let $F:\R^n\to\R$ be differentiable and,
for an iterate $x_k$, set
$$
g_k:=\nabla F(x_k).
$$
Given a symmetric positive definite Hessian approximation $B_k\succ0$, the
search direction $d_k$ is defined by
$$
B_k d_k=-g_k.
$$
An exact line search chooses
\begin{equation*}
  \alpha_k\in\arg\min_{\alpha\ge0}F(x_k+\alpha d_k),
  \qquad
  x_{k+1}=x_k+\alpha_k d_k.
\end{equation*}
With
$$
s_k:=x_{k+1}-x_k,
  \qquad
  y_k:=g_{k+1}-g_k,
$$
the BFGS update is
\begin{equation}\label{eq:bfgs-update}
  B_{k+1}
  =B_k-\frac{B_k s_k s_k^\trans B_k}{s_k^\trans B_k s_k}
       +\frac{y_k y_k^\trans}{s_k^\trans y_k}.
\end{equation}
Whenever $s_k^\trans y_k>0$, the update is well defined and preserves positive
definiteness.

\begin{definition}[Q-order]\label{def:qorder}
Let $\{x_k\}_{k\ge0}$ converge to $x_\star$, and assume $x_k\ne x_\star$ for all
sufficiently large $k$. Following the upper-bound convention of
Jay \cite{Jay2001}, the sequence has \emph{Q-order at least $p$}, where $p\ge1$,
if there exist constants $C>0$ and $k_0$ such that
\begin{equation}\label{eq:qorder-definition}
  \norm{x_{k+1}-x_\star}
  \le C\norm{x_k-x_\star}^{p}
  \qquad(k\ge k_0).
\end{equation}
Equivalently,
$$
\limsup_{k\to\infty}
  \frac{\norm{x_{k+1}-x_\star}}
       {\norm{x_k-x_\star}^{p}}<\infty.
$$
Whenever the set of admissible exponents is nonempty, the \emph{Q-order} of
$\{x_k\}_{k\ge0}$ at $x_\star$ is defined as
\begin{equation*}
  \sup\Bigl\{p\ge1:\ \{x_k\}_{k\ge0}\text{ has Q-order at least }p
  \text{ at }x_\star\Bigr\}.
\end{equation*}
The sequence converges \emph{Q-superlinearly} to $x_\star$ if
\begin{equation*}
  \frac{\norm{x_{k+1}-x_\star}}
       {\norm{x_k-x_\star}}\longrightarrow0.
\end{equation*}
Thus, for a fixed $p>1$, the estimate \eqref{eq:qorder-definition} must hold
for all sufficiently large $k$, with a constant $C$ independent of $k$.
\end{definition}
\begin{proposition}[Universal lower bound]\label{prop:lower-bound}
Let $F:\R^n\to\R$ be $C^2$ in a neighborhood of $x_\star$, with
$\nabla F(x_\star)=0$ and $\nabla^2F(x_\star)\succ0$. Let
$\{x_k\}_{k\ge0}\subset\R^n$ satisfy $x_k\to x_\star$ and
$x_k\ne x_\star$ for all sufficiently large $k$. Suppose that, for every
$k\ge0$, there exist $d_k\in\R^n$ and $\alpha_k\ge0$ such that
$$
  \alpha_k\in\arg\min_{\alpha\ge0}F(x_k+\alpha d_k),
  \qquad
  x_{k+1}=x_k+\alpha_kd_k.
$$
Then $\{x_k\}_{k\ge0}$ has Q-order at least one at $x_\star$.
\end{proposition}
\begin{proof}
By continuity and positive definiteness of the Hessian, there are a
neighborhood $U$ of $x_\star$ and constants $0<m\le M$ such that
$$
\frac m2\norm{x-x_\star}^2
  \le F(x)-F(x_\star)
  \le \frac M2\norm{x-x_\star}^2
  \qquad(x\in U).
$$
For all sufficiently large $k$, both $x_k$ and $x_{k+1}$ lie in $U$.
Since $\alpha=0$ is feasible in the line search,
$F(x_{k+1})\le F(x_k)$, and hence
$$
\norm{x_{k+1}-x_\star}
  \le \sqrt{M/m}\,\norm{x_k-x_\star}
$$
for every sufficiently large $k$. Hence \eqref{eq:qorder-definition} holds
with $p=1$ and $C=\sqrt{M/m}$.
\end{proof}
\begin{lemma}[Identity-initialized realization on a two-dimensional subspace]
\label{lem:identity-realization}
Let $V\subset\R^n$ be a two-dimensional linear subspace, and let
$F:\R^n\to\R$ be differentiable and strongly convex. Consider a sequence of
distinct points $\{x_k\}_{k\ge0}\subset V$, and define
$$
g_k:=\nabla F(x_k),\qquad
  s_k:=x_{k+1}-x_k,\qquad
  y_k:=g_{k+1}-g_k
  \qquad(k\ge0).
$$
Assume $g_k\in V$ for every $k\ge0$ and
\begin{subequations}
\renewcommand{\theequation}{\theparentequation.\arabic{equation}}
\begin{align}
  g_{k+1}^\trans s_k&=0
  \qquad(k\ge0),\label{eq:realization-orthogonality}\\
  g_{k+2}&\in\Span\{y_k\}\setminus\{0\}
  \qquad(k\ge0).\label{eq:realization-span}
\end{align}
\end{subequations}
Assume in addition that
\begin{equation*}
  s_0=-\beta_0g_0
  \qquad\text{for some }\beta_0>0.
\end{equation*}
Then BFGS with $B_0=I_n$ and exact line search generates exactly
$\{x_k\}_{k\ge0}$.
Every $B_k$ is positive definite and, with respect to the orthogonal
decomposition $\R^n=V\oplus V^\perp$, has the block form
$$
B_k=B_k^V\oplus I_{V^\perp},
$$
where $B_k^V$ denotes the restriction of $B_k$ to $V$.
\end{lemma}
\begin{proof}
At $k=0$,
$B_0s_0=s_0=-\beta_0g_0$ and
$B_0=I_V\oplus I_{V^\perp}$. Suppose inductively that
\begin{equation*}
  B_ks_k=-\beta_kg_k,
  \qquad \beta_k>0,
  \qquad B_k=B_k^V\oplus I_{V^\perp}.
\end{equation*}
The corresponding BFGS search direction satisfies
$$
d_k=-B_k^{-1}g_k=\frac1{\beta_k}s_k.
$$
Define the line-search function $\phi_k(\alpha):=F(x_k+\alpha d_k)$. Since
$s_k=\beta_kd_k$, we have $x_{k+1}=x_k+\beta_kd_k$. Equation
\eqref{eq:realization-orthogonality} then gives
$$
\phi_k'(\beta_k)
  =\nabla F(x_k+\beta_kd_k)^\trans d_k
  =g_{k+1}^\trans d_k
  =\frac{1}{\beta_k}g_{k+1}^\trans s_k
  =0.
$$
Since the points $x_k$ are distinct, $d_k\ne0$, and the strong convexity of
$F$ implies that $\phi_k$ is strictly convex. Thus $\beta_k$ is the unique
minimizer of $\phi_k$ on $\R$. Since $\beta_k>0$, it is also the unique
minimizer over $\alpha\ge0$, and exact line search gives
$\alpha_k=\beta_k$.
Strong convexity gives $s_k^\trans y_k>0$. Hence the BFGS update is well
defined, preserves positive definiteness, and satisfies the secant relation
$B_{k+1}s_k=y_k$. Since $s_k,y_k\in V$ and $B_k$ has the stated block form,
both rank-one corrections in \eqref{eq:bfgs-update} act only on $V$. Hence
\begin{equation*}
  B_{k+1}=B_{k+1}^V\oplus I_{V^\perp}.
\end{equation*}
Combining \eqref{eq:realization-span} with
\eqref{eq:realization-orthogonality} at index $k+1$ gives
$y_k^\trans s_{k+1}=0$. Hence
$$
s_k^\trans B_{k+1}s_{k+1}
  =y_k^\trans s_{k+1}=0.
$$
Both $B_{k+1}s_{k+1}$ and $g_{k+1}$ lie in the one-dimensional space
$V\cap s_k^\perp$ and hence are collinear. Since the points in the sequence
are distinct, $s_{k+1}\ne0$. Together with $B_{k+1}\succ0$, this gives
$$
s_{k+1}^\trans B_{k+1}s_{k+1}>0.
$$
Moreover, strong convexity of $F$ and
\eqref{eq:realization-orthogonality} at index $k+1$ give
$$
0<s_{k+1}^\trans(g_{k+2}-g_{k+1})
   =-g_{k+1}^\trans s_{k+1}.
$$
These two inequalities, together with collinearity, imply
$$
B_{k+1}s_{k+1}=-\beta_{k+1}g_{k+1}
  \qquad\text{for some }\beta_{k+1}>0.
$$
This closes the induction.
\end{proof}
\section{An alternating gradient construction compatible with BFGS}
\label{sec:discrete}
We first carry out the construction in $\R^2$. Let
$$
J:=\begin{pmatrix}0&-1\\1&0\end{pmatrix},
$$
which is the counterclockwise rotation by $\pi/2$. For each nonzero vector
$g_k$, define
\begin{equation}\label{eq:tangent-data}
  r_k:=\norm{g_k}>0,\qquad
  e_k:=\frac{g_k}{r_k},\qquad
  q_k:=Je_k,\qquad
  \Delta_k:=\delta_kq_k,
\end{equation}
where $\delta_k\in\R$ is a scalar parameter to be specified later in the
construction. Thus $\Delta_k$ is a tangential perturbation of $g_k$, with its
direction and magnitude determined by $\delta_k$. The pair $(e_k,q_k)$ is an
orthonormal basis. Consequently, $g_k^\trans\Delta_k=0$ for every choice of
$\delta_k$.
For each $k\ge0$, define the candidate BFGS iterate by
\begin{equation*}
  x_k:=g_k+\Delta_k.
\end{equation*}
The smooth realization in Section~\ref{sec:smooth} will ensure that
$\nabla F(x_k)=g_k$.
For $k\ge1$, whenever $g_k\ne g_{k-1}$, define
\begin{equation*}
  v_k:=g_k-g_{k-1},\qquad
  u_k:=\frac{v_k}{\norm{v_k}},\qquad
  P_k:=u_k^\trans g_k,\qquad
  D_k:=u_k^\trans q_k,
\end{equation*}
and
\begin{equation*}
  S_k:=\bigl|\det(e_{k-1},e_k)\bigr|\in[0,1].
\end{equation*}
Once $\delta_k$ has been chosen, define
\begin{equation}\label{eq:parameterized-recurrence}
  \tau_k:=P_k+D_k\delta_k,
  \qquad
  g_{k+1}:=\tau_ku_k.
\end{equation}
Whenever $g_{k+1}\ne0$, the next perturbation has the same form
$\Delta_{k+1}=\delta_{k+1}q_{k+1}$, with $\delta_{k+1}$ to be chosen at the
next step.
\begin{lemma}[Pre-step and post-step identities]
\label{lem:recurrence-identities}
Fix $k\ge1$ and suppose that $g_{k-1}$ and $g_k$ are distinct and nonzero.
Assume the pre-step condition
\begin{equation}\label{eq:prestep-relation}
  v_k^\trans g_k=g_k^\trans\Delta_{k-1}.
\end{equation}
Set $\omega_k:=\det(e_{k-1},e_k)$. Then
\begin{equation*}
  P_k=\frac{\delta_{k-1}r_k\omega_k}{\norm{v_k}},
  \qquad
  D_k=\frac{r_{k-1}\omega_k}{\norm{v_k}}.
\end{equation*}
Consequently, if $S_k>0$, then $D_k\ne0$ and
\begin{equation}\label{eq:ratio-identity}
  \left|\frac{P_k}{D_k}\right|
  =\frac{|\delta_{k-1}|r_k}{r_{k-1}}.
\end{equation}
Suppose further that $\delta_k$ is chosen so that $\tau_k\ne0$, and define
$g_{k+1}$ by \eqref{eq:parameterized-recurrence}. For any tangential choice
$\Delta_{k+1}\perp g_{k+1}$, set $x_{k+1}:=g_{k+1}+\Delta_{k+1}$. Then
\begin{equation*}
  g_{k+1}^\trans(x_{k+1}-x_k)=0,
  \qquad
  g_{k+1}\in\Span\{g_k-g_{k-1}\}\setminus\{0\}.
\end{equation*}
\begin{equation}\label{eq:angular-identities}
  S_{k+1}=\frac{r_{k-1}S_k}{\norm{g_k-g_{k-1}}},
  \qquad
  |D_k|=S_{k+1}.
\end{equation}
\end{lemma}
\begin{proof}
By \eqref{eq:prestep-relation},
$$
P_k=\frac{v_k^\trans g_k}{\norm{v_k}}
  =\frac{g_k^\trans\Delta_{k-1}}{\norm{v_k}}
  =\frac{\delta_{k-1}r_k\omega_k}{\norm{v_k}}.
$$
Direct calculation also gives
$$
D_k=\frac{v_k^\trans q_k}{\norm{v_k}}
  =\frac{r_{k-1}\omega_k}{\norm{v_k}}.
$$
Since $S_k=|\omega_k|$, the claims concerning $D_k$ and
\eqref{eq:ratio-identity} follow.

Now suppose that $\tau_k\ne0$. Using
$g_{k+1}^\trans\Delta_{k+1}=0$ and
\eqref{eq:parameterized-recurrence},
\begin{align*}
  g_{k+1}^\trans(x_{k+1}-x_k)
  &=\tau_ku_k^\trans
    \bigl(\tau_ku_k+\Delta_{k+1}-g_k-\delta_kq_k\bigr)\\
  &=\tau_k(\tau_k-P_k-D_k\delta_k)=0.
\end{align*}
Since $\tau_k\ne0$ and $u_k=v_k/\norm{v_k}$,
\eqref{eq:parameterized-recurrence} yields
$$
g_{k+1}=\frac{\tau_k}{\norm{v_k}}v_k
  \in\Span\{g_k-g_{k-1}\}\setminus\{0\},
$$
which proves the span relation. Moreover,
$$
e_{k+1}=\operatorname{sgn}(\tau_k)\frac{v_k}{\norm{v_k}}.
$$
Consequently,
\begin{align*}
  S_{k+1}
  &=\bigl|\det(e_k,e_{k+1})\bigr|
    =\frac{\bigl|\det(e_k,v_k)\bigr|}{\norm{v_k}}
    =\frac{r_{k-1}\bigl|\det(e_{k-1},e_k)\bigr|}{\norm{v_k}}
    =\frac{r_{k-1}S_k}{\norm{g_k-g_{k-1}}},\\
  |D_k|
  &=\frac{r_{k-1}|\omega_k|}{\norm{v_k}}
    =\frac{r_{k-1}S_k}{\norm{g_k-g_{k-1}}}
    =S_{k+1}.
\end{align*}
These are precisely the two identities in
\eqref{eq:angular-identities}.
\end{proof}
The post-step identities \eqref{eq:angular-identities} also show that the
recurrence remains nondegenerate. Along any nonzero recurrence, $S_1>0$
implies $S_k>0$ for every $k$. Moreover, with
$\mu_k:=r_k/r_{k-1}$,
\begin{equation}\label{eq:angle-product}
  S_{k+1}
  =\frac{S_k}{\norm{e_{k-1}-\mu_ke_k}}
  \ge\frac{S_k}{1+\mu_k}.
\end{equation}
Consequently, if $\sum_{k\ge1}\mu_k<\infty$, then iteration of
\eqref{eq:angle-product} and the inequality $\log(1+t)\le t$ give
$$
  S_{k+1}
  \ge S_1\prod_{\ell=1}^k(1+\mu_\ell)^{-1}
  \ge S_1\exp\left(-\sum_{\ell=1}^{\infty}\mu_\ell\right)
  =:D_{\min}>0.
$$
Since $|D_k|=S_{k+1}$ by \eqref{eq:angular-identities}, both $S_k$ and
$|D_k|$ are uniformly bounded away from zero.
\begin{lemma}[Alternating scale construction]\label{lem:alternating}
For every $\sigma\in(0,1)$, there exist nonzero vectors
$\{g_k\}_{k\ge0}$, scalars $\{\delta_k\}_{k\ge0}$, and constants $a,b>0$
satisfying
\eqref{eq:tangent-data}--\eqref{eq:parameterized-recurrence} such that
\begin{align}
  &r_k\downarrow0,
  \qquad
  \frac{r_{k+1}}{r_k}\longrightarrow0,
  \qquad
  \sum_{k=0}^{\infty}\frac{r_{k+1}}{r_k}<\infty,
  \notag\\
  &|\delta_k|=o(r_k^m)
  \qquad\text{for every fixed integer }m\ge0.
  \label{eq:all-order-flatness}
\end{align}
Moreover, the odd transitions have asymptotic logarithmic order one:
\begin{equation}\label{eq:log-ratio-subsequence}
  \frac{-\log r_{2j+2}}{-\log r_{2j+1}}\longrightarrow1
  \qquad(j\to\infty).
\end{equation}
In addition, the following uniform bounds hold:
\begin{equation}\label{eq:uniform-smallness}
  \frac{r_{k+1}}{r_k}\le\frac14\quad(k\ge0),
  \qquad
  \sup_{k\ge0}\frac{|\delta_k|}{r_k}\le\sigma.
\end{equation}
The initial values are given by
\begin{equation*}
  g_0=(a,0),\qquad
  g_1=(0,b),\qquad
  \delta_0=b,\qquad
  0<|\delta_1|<\frac a2.
\end{equation*}
\end{lemma}
The proof is an induction over pairs of steps. At an even index $k=2j$, the
past data determine the cancellation scale $A_k=|P_k/D_k|$. We first choose
the next radius $r_{k+1}=|\tau_k|$ and then solve
$\delta_k=(\tau_k-P_k)/D_k$; this makes the new radius arbitrarily small. At
the following odd index, the past data determine the nonzero leading term
$T_k=|P_{k+1}|$. We choose $\delta_{k+1}$ too small to cancel this term, so
$r_{k+2}$ lies between fixed positive multiples of $T_k$. The summable
sequence $\{\eta_\ell\}$ budgets the two adjacent radius ratios, while $c_j$
makes the bound on $A_k$ propagate from $k$ to $k+2$. Table~\ref{tab:even-odd}
summarizes the two choices; no later choice changes data fixed at an earlier
step.
\begin{table}[!ht]
\centering
\caption{Even--odd choice order in the two-step induction.}
\label{tab:even-odd}
\begin{tabular}{@{}p{0.17\textwidth}p{0.42\textwidth}p{0.33\textwidth}@{}}
\toprule
Step & Choice & Consequence \\
\midrule
Even $k=2j$
  & Choose nonzero $\tau_k$ small and set
    $\delta_k=(\tau_k-P_k)/D_k$.
  & $r_{k+1}=|\tau_k|$ can be made arbitrarily small. \\
Odd $k+1$
  & Choose $\delta_{k+1}$ so that
    $|D_{k+1}\delta_{k+1}|\le T_k/2$.
  & $\tfrac12T_k\le r_{k+2}\le\tfrac32T_k$. \\
\bottomrule
\end{tabular}
\par\smallskip
\begin{minipage}{0.94\textwidth}
\small\emph{Note.}
At each row, the quantities preceding the choice are fixed by earlier steps.
\end{minipage}
\end{table}
\begin{proof}
Choose $\vartheta\in(0,1)$ sufficiently small that
$\tfrac23\vartheta\cdot2^{-5}\le\sigma$, and set
$\eta_\ell:=\vartheta\cdot2^{-\ell-4}$ for $\ell\ge1$. Choose $a\in(0,1)$ with
$a^2\le\sigma$, and then choose $b>0$ such that
$$
\frac ba\le\min\{\sigma,1/8\}.
$$
Set $g_0=(a,0)$, $g_1=(0,b)$, and $\delta_0=b$. Then $S_1=1$, and the
initial pre-step relation
$$
v_1^\trans g_1=g_1^\trans\Delta_0
$$
holds. We first record how this relation propagates. Suppose that the
hypotheses of Lemma~\ref{lem:recurrence-identities} hold at an index $\ell$
and that $\tau_\ell\ne0$. That lemma gives
$$
S_{\ell+1}
  =\frac{r_{\ell-1}S_\ell}{\norm{g_\ell-g_{\ell-1}}},
  \qquad
  g_{\ell+1}^\trans(x_{\ell+1}-x_\ell)=0.
$$
Since
$x_{\ell+1}-x_\ell=v_{\ell+1}+\Delta_{\ell+1}-\Delta_\ell$ and
$g_{\ell+1}^\trans\Delta_{\ell+1}=0$, the second equality becomes
$$
v_{\ell+1}^\trans g_{\ell+1}
  =g_{\ell+1}^\trans\Delta_\ell,
$$
which is the pre-step condition at index $\ell+1$. In particular, if
$S_\ell>0$, then $S_{\ell+1}>0$ and the pre-step condition propagates to the
next index.
Let
$\kappa:=(a^2+b^2)^{1/2}$. At $k=1$,
$$
P_1=\frac{b^2}{\kappa},\qquad
  D_1=\frac a\kappa,\qquad
  r_2=\frac{|b^2+a\delta_1|}{\kappa}.
$$
For all sufficiently small $\delta_1$, we have $\tau_1\ne0$. Since $S_1=1$,
the preceding propagation argument at index $1$ gives $S_2>0$ and the
pre-step condition at index $2$. As $\delta_1\to0$ through nonzero values,
$r_2\to b^2/\kappa>0$, while \eqref{eq:ratio-identity} at $k=2$ gives
$$
A_2:=\left|\frac{P_2}{D_2}\right|
  =|\delta_1|\frac{r_2}{b}\longrightarrow0.
$$
Moreover, $r_2/r_1\to b/\kappa<1/4$ because $b/a\le1/8$ and
$\kappa\ge a$. Hence all the following requirements hold simultaneously for
every sufficiently small nonzero $\delta_1$.
Choose $\delta_1\ne0$ sufficiently small to satisfy
\begin{equation*}
  \frac{|\delta_1|}{b}\le\sigma,
  \qquad |\delta_1|<\frac a2,
  \qquad \frac{r_2}{r_1}\le\frac14,
\end{equation*}
and
$$
0<A_2\le\min\left\{r_2^4,\frac49\eta_3r_2\right\}.
$$
At every even index $k=2j\ge2$, with $j\ge1$, maintain
\begin{equation}\label{eq:induction-invariant}
  0<A_k:=\left|\frac{P_k}{D_k}\right|
  \le\min\left\{r_k^{j+3},\frac49\eta_{k+1}r_k\right\}.
\end{equation}
The last inequality is \eqref{eq:induction-invariant} at $k=2$ and $j=1$,
and hence establishes the base case. At an even index
$k=2j\ge2$, choose a nonzero signed number $\tau_k$ such that
\begin{equation}\label{eq:small-even-step}
  |\tau_k|\le
  \min\left\{\frac12|P_k|,r_k^{j+2},\eta_kr_k\right\}
\end{equation}
and
\begin{equation}\label{eq:log-choice}
  \frac{1+|\log(A_k/r_k)|}{|\log|\tau_k||}\le\frac1j.
\end{equation}
These conditions are compatible because $|\tau_k|$ can be chosen
arbitrarily small. Set
\begin{equation*}
  \delta_k=\frac{\tau_k-P_k}{D_k}.
\end{equation*}
Then $r_{k+1}=|\tau_k|$ and
\begin{equation}\label{eq:delta-comparison}
  \frac12A_k\le|\delta_k|\le\frac32A_k.
\end{equation}
Since $S_k>0$, the propagation argument above gives $S_{k+1}>0$ and the
pre-step condition at index $k+1$. At this next odd index, equation
\eqref{eq:ratio-identity} gives
$T_k:=|P_{k+1}|>0$. Choose $\delta_{k+1}\ne0$ sufficiently small to satisfy
\begin{equation}\label{eq:odd-perturbation}
  |D_{k+1}\delta_{k+1}|\le\frac12T_k,
  \qquad
  |\delta_{k+1}|\le r_{k+1}^{j+2},
  \qquad
  |\delta_{k+1}|\le c_jr_{k+1}T_k^{j+3},
\end{equation}
where
\begin{equation}\label{eq:cj-choice}
  0<c_j\le\min\left\{2^{-(j+3)},\frac49\eta_{k+3}\right\}.
\end{equation}
The right-hand sides in \eqref{eq:odd-perturbation} are positive, so a nonzero
$\delta_{k+1}$ satisfying all three bounds exists.
Set $\tau_{k+1}=P_{k+1}+D_{k+1}\delta_{k+1}$. The first inequality in
\eqref{eq:odd-perturbation} prevents cancellation of the leading term. Hence
\begin{equation}\label{eq:odd-two-sided}
  \frac12T_k\le r_{k+2}\le\frac32T_k.
\end{equation}
Since $S_{k+1}>0$, the propagation argument at index $k+1$ gives
$S_{k+2}>0$ and the pre-step condition at index $k+2$.
We now verify that the induction closes. By \eqref{eq:ratio-identity},
$$
T_k=|D_{k+1}|\,|\delta_k|\frac{r_{k+1}}{r_k}.
$$
Since $|D_{k+1}|\le1$, equations \eqref{eq:delta-comparison},
\eqref{eq:odd-two-sided}, and \eqref{eq:induction-invariant} imply
\begin{equation}\label{eq:odd-ratio-bound}
  \frac{r_{k+2}}{r_{k+1}}
  \le\frac94\frac{A_k}{r_k}
  \le\eta_{k+1}.
\end{equation}
At the next even index, \eqref{eq:ratio-identity} and
\eqref{eq:odd-perturbation} give
$$
A_{k+2}
  =|\delta_{k+1}|\frac{r_{k+2}}{r_{k+1}}
  \le c_jT_k^{j+3}r_{k+2}.
$$
Since $T_k\le2r_{k+2}$ and
$T_k=|u_{k+1}^\trans g_{k+1}|\le r_{k+1}<1$, the two bounds in
\eqref{eq:cj-choice} yield
$$
A_{k+2}
  \le\min\left\{r_{k+2}^{j+4},\frac49\eta_{k+3}r_{k+2}\right\},
$$
which is precisely \eqref{eq:induction-invariant} at the next even index.
At even indices, \eqref{eq:small-even-step} gives
$r_{k+1}/r_k\le\eta_k$, while at the following odd indices
\eqref{eq:odd-ratio-bound} gives $r_{k+2}/r_{k+1}\le\eta_{k+1}$.
The summability of $\{\eta_\ell\}_{\ell\ge1}$ implies the summability of the
adjacent ratios. By construction, all these ratios are at most $1/4$. Hence
\eqref{eq:angle-product} yields a constant
$D_{\min}>0$ such that
\begin{equation}\label{eq:D-uniform}
  0<D_{\min}\le|D_k|\le1\qquad(k\ge1).
\end{equation}
For $j\ge1$, the first bound in \eqref{eq:induction-invariant},
\eqref{eq:delta-comparison}, and the middle inequality in
\eqref{eq:odd-perturbation} give
$$
|\delta_{2j}|\le\frac32r_{2j}^{j+3},
  \qquad
  |\delta_{2j+1}|\le r_{2j+1}^{j+2}.
$$
For any fixed integer $m\ge0$, division by the corresponding powers of the
radii gives
$$
  \frac{|\delta_{2j}|}{r_{2j}^m}
  \le\frac32r_{2j}^{j+3-m}\longrightarrow0,
  \qquad
  \frac{|\delta_{2j+1}|}{r_{2j+1}^m}
  \le r_{2j+1}^{j+2-m}\longrightarrow0.
$$
Indeed, each exponent on the right is at least one for all sufficiently large
$j$, while $r_k\to0$. The even and odd subsequences therefore establish
\eqref{eq:all-order-flatness}. The prescribed $\sigma$ bound holds at $k=0,1$.
At an even index,
$$
\frac{|\delta_k|}{r_k}
  \le\frac23\eta_{k+1}\le\sigma.
$$
At the following odd index,
$$
\frac{|\delta_{k+1}|}{r_{k+1}}
  \le r_{k+1}^{j+1}\le a^2\le\sigma,
$$
where the penultimate inequality uses $j\ge1$ and $r_{k+1}\le r_0=a<1$.
This gives \eqref{eq:uniform-smallness}.
It remains to prove the logarithmic limit. For every even $k=2j\ge2$,
\eqref{eq:D-uniform}, \eqref{eq:ratio-identity}, and
\eqref{eq:odd-two-sided} give the uniform two-step estimate
\begin{equation*}
  r_{k+2}\asymp r_{k+1}\frac{|\delta_k|}{r_k}.
\end{equation*}
Here $a_k\asymp b_k$ means that $a_k/b_k$ is bounded above and below by
positive constants independent of $k$. Taking negative logarithms and using
\eqref{eq:delta-comparison},
$$
-\log r_{k+2}
  =-\log r_{k+1}-\log(|\delta_k|/r_k)+O(1),
$$
where the $O(1)$ term is uniform. By \eqref{eq:delta-comparison},
$\log(|\delta_k|/r_k)$ differs from $\log(A_k/r_k)$ by a uniformly bounded
quantity. Since $k=2j\ge2$ and $r_{k+1}=|\tau_k|<1$,
$$
\frac{-\log r_{k+2}}{-\log r_{k+1}}
  =1+
  \frac{-\log(A_k/r_k)+O(1)}{|\log|\tau_k||}.
$$
The uniform $O(1)$ term is bounded in absolute value by a constant $C>0$
independent of $j$. Hence \eqref{eq:log-choice} gives
$$
\left|
    \frac{-\log r_{k+2}}{-\log r_{k+1}}-1
  \right|
  \le
  \frac{|\log(A_k/r_k)|+C}{|\log|\tau_k||}
  \le \frac{\max\{1,C\}}{j}
  \longrightarrow0.
$$
This proves \eqref{eq:log-ratio-subsequence}.
\end{proof}
\begin{remark}
The construction alternates between cancellation and retention. A
near-cancellation step makes $r_{k+1}$ very small, while the next step retains
the leading term. Hence every adjacent ratio tends to zero, whereas along these
odd transitions the ratio of successive negative logarithms tends to one.
These two properties yield Q-superlinear convergence with Q-order one.
\end{remark}

\section{Compactly supported smooth realization and convex conjugacy}
\label{sec:smooth}
Fix $n\ge2$. Choose a linear isometric embedding
$\iota:\R^2\to\R^n$ and set $V:=\iota(\R^2)$. Apply $\iota$ to the sequences
constructed in Section~\ref{sec:discrete}. To simplify notation, we continue
to denote the images $\iota(g_k)$, $\iota(\Delta_k)$, and $\iota(x_k)$ by
$g_k$, $\Delta_k$, and $x_k$, respectively. Thus all these vectors lie in the
fixed two-dimensional subspace $V\subset\R^n$.
\begin{lemma}[Compactly supported interpolation]\label{lem:interpolation}
There exists a constant $C_n>0$, depending only on $n$ and on a fixed smooth
cutoff. For any sequences $\{g_k\}_{k\ge0}$, $\{\delta_k\}_{k\ge0}$, and
$\{\Delta_k\}_{k\ge0}$ constructed in Lemma~\ref{lem:alternating}, let
$x_k:=g_k+\Delta_k$. If \eqref{eq:uniform-smallness} holds, then there exists
a function $H\in C^\infty(\R^n)$ such that
\begin{align*}
  &H-Q\in C_c^\infty(\R^n),
  \qquad
  \supp(H-Q)\subset\overline{B}(0,5r_0/4),
  \\
  &\nabla H(g_k)=x_k\quad(k\ge0),
  \qquad
  \nabla H(0)=0,
  \qquad
  \nabla^2H(0)=I_n,
  \\
  &\sup_{z\in\R^n}\norm{\nabla^2H(z)-I_n}_{\op}
  \le C_n\sup_{k\ge0}\frac{|\delta_k|}{r_k}.
\end{align*}
\end{lemma}
\begin{proof}
Choose a smooth cutoff function $\psi\in C_c^\infty(\R^n)$ such that
$\supp\psi\subset B(0,1)$ and $\psi$ is identically one in a neighborhood of
the origin.
Let
$$
E=\{0\}\cup\{g_k: k\ge0\},
  \qquad
  \ell_k=\operatorname{dist}\bigl(g_k,E\setminus\{g_k\}\bigr).
$$
Since $r_{k+1}/r_k\le1/4$, radial separation gives
$$
\frac34r_k\le\ell_k\le r_k.
$$
Indeed, every later point has norm at most $r_{k+1}\le r_k/4$, so its distance
from $g_k$ is at least $3r_k/4$. For $k=0$, there are no earlier points. For
$k\ge1$, every earlier point has norm at least $r_{k-1}\ge4r_k$, so its
distance from $g_k$ is at least $3r_k$. The origin is at distance $r_k$ from
$g_k$, which also gives the upper bound on $\ell_k$.
Set $\rho_k=\ell_k/4$. Then
\begin{equation*}
  \frac{3}{16}r_k\le\rho_k\le\frac14r_k.
\end{equation*}
For $i\ne j$, the definition of $\ell_i$ and $\ell_j$ gives
$$
\rho_i+\rho_j
  \le\frac14\norm{g_i-g_j}+\frac14\norm{g_i-g_j}
  <\norm{g_i-g_j}.
$$
Hence the closed balls $\overline{B}(g_k,\rho_k)$ are pairwise disjoint.
Moreover, for every $z\in\overline{B}(g_k,\rho_k)$, the reverse triangle
inequality gives
$$
\norm{z}\ge\norm{g_k}-\norm{z-g_k}
  \ge r_k-\rho_k\ge\frac34r_k>0.
$$
Therefore, $0\notin\overline{B}(g_k,\rho_k)$ for every $k\ge0$.
Define
\begin{equation*}
  \phi_k(z)
  =\psi\left(\frac{z-g_k}{\rho_k}\right)
   \Delta_k^\trans(z-g_k),
  \qquad
  h(z)=\sum_{k=0}^{\infty}\phi_k(z),
\end{equation*}
with $h(0)=0$. Since
$\supp\phi_k\subset\overline{B}(g_k,\rho_k)$, these supports are pairwise
disjoint, and at each nonzero point at most one summand is nonzero. For every
integer $m\ge0$, the product and chain rules yield
\begin{equation}\label{eq:bump-derivative}
  \norm{D^m\phi_k}_{L^\infty}
  \le C_m|\delta_k|\rho_k^{1-m}
  \le C_m'|\delta_k|r_k^{1-m}.
\end{equation}
For $m\ge1$, $\norm{D^m f(z)}_{\op}$ denotes the operator norm of the
$m$-linear derivative. For $m=0$, set $\norm{D^0f(z)}:=|f(z)|$. In either
case, write
$$
\norm{D^m f}_{L^\infty}
  :=\sup_{z\in\R^n}\norm{D^m f(z)},
$$
where for $m\ge1$ the norm on the right is the operator norm just defined.
Fix integers $m,L\ge0$. Points in $\supp\phi_k$ have norm comparable to
$r_k$, because
$$
\frac34r_k\le\norm{z}\le\frac54r_k
  \qquad(z\in\supp\phi_k).
$$
Since $r_k\to0$, for every $\varepsilon>0$ only finitely many supports
$\supp\phi_k$ intersect the set $\{z:\norm{z}\ge\varepsilon\}$. Hence the
family $\{\supp\phi_k\}_{k\ge0}$ is locally finite on
$\R^n\setminus\{0\}$, and $h$ is smooth there.

Let $\{z_\nu\}_{\nu\ge1}\subset\R^n\setminus\{0\}$ satisfy $z_\nu\to0$.
For each $\nu$ such that
$z_\nu\notin\bigcup_{k\ge0}\supp\phi_k$, the ratio below is zero. For the
remaining indices, disjointness gives a unique $k_\nu$ such that
$z_\nu\in\supp\phi_{k_\nu}$. Along this subsequence, the preceding
comparison implies $r_{k_\nu}\to0$, and hence $k_\nu\to\infty$. Using
\eqref{eq:bump-derivative} and the lower bound on $\norm{z_\nu}$ gives
$$
\frac{\norm{D^m h(z_\nu)}}{\norm{z_\nu}^L}
  \le C_{m,L}|\delta_{k_\nu}|r_{k_\nu}^{1-m-L}
  \longrightarrow0,
$$
where the limit follows from
$|\delta_k|=o(r_k^{L+m})$ in \eqref{eq:all-order-flatness}. Therefore,
$$
\norm{D^m h(z)}=o(\norm{z}^L)
  \qquad(z\to0),
$$
for every fixed pair of integers $m,L\ge0$. For $m\ge0$, define
$\mathcal D_m(z):=D^m h(z)$ for $z\ne0$ and $\mathcal D_m(0):=0$. Taking
$L=0$ shows that $\mathcal D_m$ is continuous at
the origin. If $D^m h=\mathcal D_m$, then the estimate with $L=1$ shows that
$\mathcal D_m$ is differentiable at the origin with derivative
$0=\mathcal D_{m+1}(0)$. On $\R^n\setminus\{0\}$, termwise differentiation
of the locally finite sum gives $D\mathcal D_m=\mathcal D_{m+1}$. Induction
yields
$h\in C^\infty(\R^n)$, $D^m h=\mathcal D_m$, and $D^m h(0)=0$ for every $m$.
Since $\psi$ is constant near zero and the supports are disjoint,
$$
\nabla h(g_k)=\Delta_k.
$$
All supports lie in $\overline{B}(0,5r_0/4)$. Hence
$h\in C_c^\infty(\R^n)$.
Taking $m=2$ in \eqref{eq:bump-derivative} and using disjointness gives
$$
\sup_{z\in\R^n}\norm{\nabla^2h(z)}_{\op}
  \le C_n\sup_{k\ge0}\frac{|\delta_k|}{r_k}.
$$
Since $x_k=g_k+\Delta_k$, setting $H=Q+h$ proves the lemma.
\end{proof}
\begin{lemma}[Conjugacy preserves the quadratic tail]\label{lem:conjugacy}
Let $H\in C^\infty(\R^n)$ satisfy
$$
H-Q\in C_c^\infty(\R^n),
  \qquad
  \sup_{z\in\R^n}\norm{\nabla^2H(z)-I_n}_{\op}\le\theta<1,
$$
with $\nabla H(0)=0$ and $\nabla^2H(0)=I_n$. Let $F:=H^*$ be the convex
conjugate of $H$:
$$
H^*(x):=\sup_{z\in\R^n}\bigl\{x^\trans z-H(z)\bigr\}.
$$
Then $\nabla H$ is a global $C^\infty$ diffeomorphism,
$F\in C^\infty(\R^n)$, and
\begin{align}
  &\nabla F=(\nabla H)^{-1},
  \qquad
  \nabla^2F(x)
  =\bigl[\nabla^2H((\nabla H)^{-1}(x))\bigr]^{-1},
  \label{eq:legendre-hessian}\\
  &\frac1{1+\theta}I_n
  \preceq\nabla^2F(x)
  \preceq\frac1{1-\theta}I_n,
  \label{eq:F-hessian-bounds}\\
  &\sup_{x\in\R^n}\norm{\nabla^2F(x)-I_n}_{\op}
  \le\frac{\theta}{1-\theta},
  \label{eq:F-hessian-close}\\
  &F-Q\in C_c^\infty(\R^n),
  \qquad
  \supp(F-Q)\subset\supp(H-Q).
  \label{eq:F-compact-support}
\end{align}
The origin is the unique minimizer of $F$, and $\nabla^2F(0)=I_n$.
\end{lemma}
\begin{proof}
For every $z,w\in\R^n$, the operator-norm bound in the assumptions gives
$$
  \left|w^\trans\bigl(\nabla^2H(z)-I_n\bigr)w\right|
  \le \norm{\nabla^2H(z)-I_n}_{\op}\norm{w}^2
  \le \theta\norm{w}^2.
$$
Equivalently,
$$
(1-\theta)I_n\preceq\nabla^2H(z)\preceq(1+\theta)I_n.
$$
For any $u,v\in\R^n$, integration along the line segment from $v$ to $u$ yields
$$
  \bigl(\nabla H(u)-\nabla H(v)\bigr)^\trans(u-v)
  =\int_0^1
    (u-v)^\trans\nabla^2H\bigl(v+t(u-v)\bigr)(u-v)\,dt
  \ge(1-\theta)\norm{u-v}^2.
$$
Thus $\nabla H$ is strongly monotone with modulus $1-\theta$ and is therefore
injective.

The lower Hessian bound and $\nabla H(0)=0$ also give
$$
  H(z)\ge H(0)+\frac{1-\theta}{2}\norm{z}^2.
$$
Consequently, for every fixed $x\in\R^n$,
$$
  H(z)-x^\trans z
  \ge H(0)+\frac{1-\theta}{2}\norm{z}^2-\norm{x}\norm{z}
  \longrightarrow\infty
  \qquad(\norm{z}\to\infty).
$$
The function $z\mapsto H(z)-x^\trans z$ is therefore coercive and strongly
convex, so it has a unique minimizer. Its first-order optimality condition is
$\nabla H(z)=x$, which proves that $\nabla H$ is surjective. Since
$\nabla^2H(z)$ is nonsingular for every $z$, the inverse function theorem shows
that the bijective map $\nabla H$ has a global $C^\infty$ inverse. The unique
maximizer in the definition of $H^*(x)$ is $z=(\nabla H)^{-1}(x)$. Hence
$$
F(x)=x^\trans(\nabla H)^{-1}(x)
       -H((\nabla H)^{-1}(x)).
$$
Differentiating this identity gives
\eqref{eq:legendre-hessian}, and matrix inversion gives
\eqref{eq:F-hessian-bounds}. If
$A=\nabla^2H((\nabla H)^{-1}(x))$, then
$$
\norm{A^{-1}-I_n}_{\op}
  =\norm{A^{-1}(I_n-A)}_{\op}
  \le\frac{\theta}{1-\theta},
$$
which proves \eqref{eq:F-hessian-close}.
Let $K=\supp(H-Q)$. If $\xi\notin K$, then $H-Q$ vanishes in a
neighborhood of $\xi$. Hence $\nabla H(\xi)=\xi$, and injectivity gives
$(\nabla H)^{-1}(\xi)=\xi$.
Consequently,
$$
F(\xi)=\xi^\trans\xi-H(\xi)=\frac12\norm{\xi}^2=Q(\xi)
  \qquad(\xi\notin K),
$$
which gives \eqref{eq:F-compact-support}. Finally, $\nabla H(0)=0$. Hence
$\nabla F(0)=0$ and $\nabla^2F(0)=I_n$. The minimizer is unique by strong
convexity.
\end{proof}
\begin{lemma}[Scaling invariance]\label{lem:scaling}
Let $F\in C^2(\R^n)$ generate a BFGS sequence $\{x_k\}_{k\ge0}$ with exact
line search and matrices $\{B_k\}_{k\ge0}$. For $\lambda>0$, define
$$
F_\lambda(x)=\lambda^2F(x/\lambda),
  \qquad x_k^\lambda=\lambda x_k.
$$
Then $\{x_k^\lambda\}_{k\ge0}$ is a valid exact-line-search BFGS run for
$F_\lambda$ from $(x_0^\lambda,B_0)$, using the same selected step lengths and
the same matrices $\{B_k\}_{k\ge0}$. Moreover,
\begin{align*}
  \nabla^2F_\lambda(x)&=\nabla^2F(x/\lambda),\\
  \supp(F_\lambda-Q)&=\lambda\supp(F-Q).
\end{align*}
If $x_k\to x_\star$ and $x_\star^\lambda:=\lambda x_\star$, then, for every
$p\ge1$, the sequence $\{x_k^\lambda\}_{k\ge0}$ has Q-order at least $p$ at
$x_\star^\lambda$ if and only if $\{x_k\}_{k\ge0}$ has Q-order at least $p$
at $x_\star$. Consequently, the two sequences have the same Q-order whenever
it is defined.
\end{lemma}
\begin{proof}
We have
$\nabla F_\lambda(x)=\lambda\nabla F(x/\lambda)$. Hence gradients, steps, and
secant differences all scale by $\lambda$. If $d_k$ is the original search
direction, then $d_k^\lambda=\lambda d_k$, and
$$
F_\lambda(x_k^\lambda+\alpha d_k^\lambda)
  =\lambda^2F(x_k+\alpha d_k).
$$
Hence each selected exact line-search minimizer for the original problem is
also an exact line-search minimizer for the scaled problem. Choosing these same
minimizers gives the stated scaled run. When $F$ is strongly convex, the
minimizer along every nonzero search direction is unique. Since both rank-one
terms in the BFGS update are homogeneous of degree zero in $(s_k,y_k)$, the
BFGS matrices are unchanged. The Hessian and support identities follow directly
from the definition. Finally,
$$
\norm{x_k^\lambda-x_\star^\lambda}
  =\lambda\norm{x_k-x_\star}.
$$
Hence scaling changes only the constant in a Q-order estimate and leaves the set
of admissible exponents unchanged.
\end{proof}
\section{Proof of the main theorem}
\label{sec:proof}
\begin{proof}[Proof of Theorem~\ref{thm:main}]
Fix $n\ge2$, $\varepsilon>0$, and $R>0$. Let $C_n$ be the constant from
Lemma~\ref{lem:interpolation}, and choose
\begin{equation*}
  0<\theta<\min\left\{\frac12,\frac{\varepsilon}{1+\varepsilon}\right\}.
\end{equation*}
Choose $\sigma\in(0,1)$ with $\sigma<\theta/C_n$, apply
Lemma~\ref{lem:alternating}, and identify the resulting two-dimensional
construction with a fixed subspace $V\subset\R^n$.
Lemma~\ref{lem:interpolation} gives
$H=Q+h$ with
$$
\sup_{z\in\R^n}\norm{\nabla^2H(z)-I_n}_{\op}<\theta,
  \qquad
  \nabla H(g_k)=x_k.
$$
Let $F:=H^*$. Lemma~\ref{lem:conjugacy} shows that $F$ is
$C^\infty$ and globally strongly convex,
$F-Q\in C_c^\infty(\R^n)$, and
$$
\sup_{x\in\R^n}\norm{\nabla^2F(x)-I_n}_{\op}
  \le\frac{\theta}{1-\theta}<\varepsilon.
$$
Moreover, $0$ is the unique minimizer and $\nabla^2F(0)=I_n$.
The interpolation identity and \eqref{eq:legendre-hessian} give
$\nabla F(x_k)=g_k$. All $x_k,g_k,s_k,y_k$ lie in $V$. The
orthogonality relation in Lemma~\ref{lem:recurrence-identities} gives
$$
g_{k+1}^\trans s_k=0\qquad(k\ge1).
$$
The same relation holds at $k=0$ by the initial choice. Indeed,
$$
x_0=(a,b,0,\ldots,0),
  \qquad
  x_1=(-\delta_1,b,0,\ldots,0).
$$
Hence $g_1^\trans(x_1-x_0)=0$. The span relation in
Lemma~\ref{lem:recurrence-identities}, with indices shifted by one, gives
$$
g_{k+2}\in\Span\{g_{k+1}-g_k\}\setminus\{0\}
  =\Span\{y_k\}\setminus\{0\}.
$$
Furthermore,
$$
s_0=x_1-x_0=(-a-\delta_1,0,\ldots,0)
  =-\left(1+\frac{\delta_1}{a}\right)g_0.
$$
Since $|\delta_1|<a/2$,
$\beta_0:=1+\delta_1/a>0$. The strictly decreasing radii imply that the
$g_k$ are distinct. Injectivity of $\nabla H$ then implies that the $x_k$ are
distinct. Hence Lemma~\ref{lem:identity-realization} applies. BFGS with
$B_0=I_n$ and exact line search generates exactly $\{x_k\}_{k\ge0}$, every $B_k$
remains positive definite, and the sequence is nonterminating.
Orthogonality of $g_k$ and $\Delta_k$ gives the exact identity
\begin{equation*}
  \norm{x_k}^2=r_k^2+\delta_k^2.
\end{equation*}
By \eqref{eq:all-order-flatness}, $|\delta_k|/r_k\to0$, and hence
\begin{equation}\label{eq:radius-transfer}
  \norm{x_k}=r_k(1+o(1)),
  \qquad
  \log\norm{x_k}=\log r_k+o(1).
\end{equation}
Together with $r_{k+1}/r_k\to0$, this implies
$$
\frac{\norm{x_{k+1}}}{\norm{x_k}}\longrightarrow0.
$$
Hence $x_k\to0$ Q-superlinearly. To make the transfer of logarithmic rates
explicit, set
$$
  R_k:=-\log r_k,
  \qquad
  L_k:=-\log\norm{x_k}.
$$
Since $r_k\to0$, we have $R_k\to\infty$, while
\eqref{eq:radius-transfer} gives $L_k=R_k+o(1)$. Both $R_{2j+1}$ and
$R_{2j+2}$ tend to infinity, so the additive $o(1)$ terms are negligible in
the following quotient. Equation~\eqref{eq:log-ratio-subsequence} therefore
implies
\begin{equation}\label{eq:primal-log-slope}
  \frac{-\log\norm{x_{2j+2}}}{-\log\norm{x_{2j+1}}}
  =\frac{L_{2j+2}}{L_{2j+1}}
  =\frac{R_{2j+2}+o(1)}{R_{2j+1}+o(1)}
  \longrightarrow1
  \qquad(j\to\infty).
\end{equation}
Suppose, to the contrary, that \eqref{eq:qorder-definition} holds for some
$p>1$ and $C>0$. Taking negative logarithms shows that, for all sufficiently
large $k$,
$$
L_{k+1}\ge pL_k-\log C.
$$
Taking $k=2j+1$ and dividing by $L_{2j+1}$ gives
$$
\liminf_{j\to\infty}\frac{L_{2j+2}}{L_{2j+1}}\ge p>1,
$$
contradicting \eqref{eq:primal-log-slope}. Hence the Q-order of
$\{x_k\}_{k\ge0}$ at the origin is one.
At this point the perturbation has compact support and the starting point is
fixed. Choose $\lambda>0$ in Lemma~\ref{lem:scaling} sufficiently small that
$$
\lambda\supp(F-Q)\cup\{\lambda x_0\}\subset B(0,R).
$$
After relabeling the scaled objective and sequence, all the preceding
properties are preserved and the required localization holds. This gives the
existence statement. The universal lower bound in
Proposition~\ref{prop:lower-bound} then shows that one is the minimum possible
Q-order and that this minimum is attained.
\end{proof}

We close with a direct iterate-level extension. Denote the right-hand side of
\eqref{eq:bfgs-update} by $B_{k+1}^{\mathrm{BFGS}}$. The
Davidon--Fletcher--Powell (DFP) update in Hessian form is
\begin{equation}\label{eq:dfp-update}
  B_{k+1}^{\mathrm{DFP}}
  =\left(I_n-\frac{y_ks_k^\trans}{y_k^\trans s_k}\right)
    B_k
    \left(I_n-\frac{s_ky_k^\trans}{s_k^\trans y_k}\right)
    +\frac{y_ky_k^\trans}{y_k^\trans s_k}.
\end{equation}
The \emph{convex Broyden class} consists of the updates
\begin{equation}\label{eq:convex-broyden-update}
  B_{k+1}^{\phi_k}
  =(1-\phi_k)B_{k+1}^{\mathrm{BFGS}}
    +\phi_k B_{k+1}^{\mathrm{DFP}},
  \qquad \phi_k\in[0,1].
\end{equation}
The parameter $\phi_k$ may vary with $k$. If $B_k\succ0$ and
$s_k^\trans y_k>0$, both endpoint updates are positive definite, so every
update in \eqref{eq:convex-broyden-update} is positive definite.

\begin{corollary}[Convex Broyden class]\label{cor:broyden}
For the objective and starting point in Theorem~\ref{thm:main}, fix any
sequence $\{\phi_k\}_{k\ge0}\subset[0,1]$. The method initialized with
$B_0=I_n$, using \eqref{eq:convex-broyden-update} and exact line search,
generates the same nonterminating point sequence as BFGS. In particular, its
iterates converge Q-superlinearly to the origin and have Q-order one.
\end{corollary}
\begin{proof}
Strong convexity gives $s_k^\trans y_k>0$ on every nonzero step. Hence all
updates in \eqref{eq:convex-broyden-update} are well defined and positive
definite. Dixon's exact-line-search equivalence
\cite{Dixon1972Identical} then shows that, from the same $x_0$ and $B_0$, the
convex Broyden updates generate the same iterates as BFGS. The rate conclusions
follow from Theorem~\ref{thm:main}.
\end{proof}
\section{Conclusion}
\label{sec:conclusion}
For exact-line-search BFGS in every finite dimension $n\ge2$, the minimum
adjacent-iterate Q-order is one, and this value is attained. The example uses
the standard initialization $B_0=I_n$. The Hessian can be made uniformly as
close to the identity as desired, and the
objective differs from a quadratic only inside an arbitrarily small ball.
Thus, under the classical regularity hypotheses for exact-line-search BFGS,
the known Q-superlinear conclusion cannot be strengthened to a uniform
adjacent-iterate power order greater than one. This boundary remains sharp
under the stronger global controls imposed here. An alternating gradient
construction produces the rate: all adjacent error ratios tend to zero, while
the logarithmic order tends to one along an infinite subsequence. Compactly
supported smooth interpolation and convex conjugacy realize this discrete
construction as a nonterminating BFGS trajectory. As a direct consequence of
Dixon's exact-line-search equivalence, the same iterate-level Q-order
conclusion holds for the convex Broyden class.
\bibliographystyle{plainnat}
\bibliography{references}
\end{document}